\documentclass[11pt]{amsart}
\usepackage{amsmath, amssymb, amscd, amsthm, mathtools, amsfonts, mathrsfs, tikz-cd, bbm, enumitem}
\usepackage{lipsum}
\usepackage{adjustbox}
\usepackage[mathscr]{euscript}
\usepackage{hyperref}
\usepackage[utf8]{inputenc}
\usepackage[english]{babel}
\usepackage{stmaryrd}
\usepackage{todonotes}
\usepackage[T1]{fontenc}
\usepackage{float}
\usepackage{arydshln}

\renewcommand{\int}{\smallint\!}

\newcommand{\op}{\mathrm{op}}
\DeclareMathOperator*{\colim}{colim}
\newcommand{\sm}{\mathrm{sm}}

\newcommand{\fp}{\mathrm{fp}}
\newcommand{\smt}{\mathrm{sm}}

\newcommand{\inv}{^{-1}}

\DeclareMathOperator{\id}{id}
\newcommand{\cdh}{\mathrm{cdh}}

\newcommand{\red}{\mathrm{red}}

\newcommand{\Mod}{\mathscr{M}\mathrm{od}}

\newcommand{\CCC}{\mathscr{C}}

\newcommand{\SSS}{\mathscr{S}}
\newcommand{\MMM}{\mathscr{M}}

\newcommand{\PPP}{\mathscr{P}}
\newcommand{\VVV}{\mathscr{V}}

\newcommand{\Sp}{{\mathscr{S}\mathrm{p}}}

\renewcommand{\Pr}{\mathscr{P}\mathrm{r}}

\newcommand{\PrLo}{{\Pr}^{\mathrm{L},\otimes}}

\newcommand{\DDD}{\mathscr{D}}

\newcommand{\Aff}{\mathscr{A}\mathrm{ff}}

\newcommand{\Lcdh}{\mathrm{L}_{\textup{cdh}}}

\DeclareMathOperator{\Fun}{Fun}
\DeclareMathOperator{\CAlg}{CAlg}
\newcommand{\PrLotimesst}{\Pr^{\mathrm{L},\otimes}_{\mathrm{st}}}
\newcommand{\dSch}{\mathrm{d}\Sch}

\newcommand{\Schfp}{\Sch^{\fp}}

\newcommand{\SchfpB}{\Schfp_{B}}
\newcommand{\SchfpBop}{\Sch^{\fp,\op}_{B}}

\DeclareMathOperator{\map}{map}

\newcommand{\N}{\mathbb{N}}
\newcommand{\Z}{\mathbb{Z}}
\newcommand{\Q}{\mathbb{Q}}
\newcommand{\A}{\mathbb{A}}

\newcommand{\E}{\mathbb{E}}
\newcommand{\Sm}{\mathscr{S}\mathrm{m}}

\newcommand{\SH}{\mathrm{S}\mathscr{H}}

\newcommand{\Sh}{\mathscr{S}\mathrm{h}}
\newcommand{\G}{\mathbb{G}}

\newcommand{\oO}{\mathcal{O}}

\DeclareMathOperator{\Spec}{Spec}

\newcommand{\Sch}{\mathscr{S}\mathrm{ch}}

\newcommand{\KM}{\mathrm{K}\mathscr{M}}

\newcommand{\Ee}{\mathrm{E}}
\newcommand{\Eu}{\underline{\mathrm{E}}}

\DeclareMathOperator{\Hom}{Hom}
\newcommand{\Homu}{\underline{\Hom}}
\DeclareMathOperator{\Frac}{Frac}

\DeclareMathOperator{\K}{K}
\DeclareMathOperator{\Ge}{G}
\DeclareMathOperator{\KH}{KH}
\DeclareMathOperator{\KGL}{KGL}

\DeclareMathOperator{\GGL}{GGL}

 \renewcommand{\to}[1][]{\overset{#1}{\rightarrow}}		
	
\newcommand{\inj}[1][]{\overset{#1}{\hookrightarrow}}		

\theoremstyle{definition}
\newtheorem{Def}{Definition}[section]
\newtheorem{Not}[Def]{Notation}
\newtheorem{Rem}[Def]{Remark}
\newtheorem{Rem*}[]{Remark}

\newtheorem{Cons}[Def]{Construction}

\newtheorem*{Conj*}{Conjecture}

\theoremstyle{plain}
\newtheorem{Prop}[Def]{Proposition}
\newtheorem{Thm}[Def]{Theorem}

\newtheorem{Lem}[Def]{Lemma}
\newtheorem{Cor}[Def]{Corollary}

    \newcounter{zaehler}
    
	\newtheorem{introthm}[zaehler]{Theorem}

\title{Duality for KGL-modules in motivic homotopy theory}

\author{Christian Dahlhausen}
\address{Institut für Mathematik, Universität Heidelberg, Im Neuenheimer Feld 205, 69121 Heidelberg, Germany}
\email{cdahlhausen@mathi.uni-heidelberg.de}

\author{Jeroen Hekking}
\address{Fakultät für Mathematik, Universität Regensburg, 93040 Regensburg, Germany}
\email{jeroen.hekking@ur.de}

\author{Storm Wolters}
\address{Mathematisch Instituut, Universiteit Leiden, Einsteinweg 55, 2333 CC Leiden, The Netherlands}
\email{s.wolters@math.leidenuniv.nl}

\begin{document}
\begin{abstract}
We prove a duality statement on modules over KH-theory in the stable motivic homotopy category whose dualizing object is given by G-theory, over any quasi-excellent scheme of characteristic zero.
\end{abstract}
\maketitle

\section*{Introduction}
Dualities play a ubiquitous role in mathematics. In its simplest form, a duality on a closed symmetric monoidal category $\CCC$ is given by a \emph{dualizing object}, i.e., an object $D$ of $\CCC$ such that the induced functor $\Homu_\CCC(-,D) \colon \CCC^\op\to\CCC$ is an equivalence of categories.
Prominent examples of such dualities are given by the functors $\Homu_k(-,k)$ on finite-dimensional $k$-vector spaces (for a field $k$) and the Pontryagin dual $\Homu(-,S^1)$ on locally compact abelian groups (where $\Homu(-,-)$ carries the compact-open topology).

Dualities in geometric contexts often appear within a six-functor formalism $\DDD$. For instance, \emph{Poincaré duality} attaches to any ``smooth'' morphism $f\colon X\to Y$ an invertible object $\omega_f$ in $\DDD(X)$ and an equivalence $f^!(-) \simeq f^*(-)\otimes\omega_f$, cf.\ \cite{ZavyalovPD}. Examples are classical Poincaré duality for closed manifolds or Poincaré duality for $\ell$-adic cohomology of schemes. 

If $\DDD$ is defined via a coefficient system over a base scheme $B$, then $\DDD$ satisfies Poincaré duality by purity. One can go further by defining---for any morphism $f\colon S\to B$ of finite type and $D \in \DDD(B)$ constructible and $\otimes$-invertible---a \emph{local duality functor} $\Homu_{\DDD(S)}(-,f^!D)$ self-adjoint on the right. Under certain assumptions this underpins a Grothendieck duality yielding that $f^!D$ is a dualizing object on the full subcategory $\DDD_c(S)$ of constructible objects in $\DDD(S)$, see \cite[Cor.~4.4.24]{CisinskiDeglise}.\footnote{More precisely, $\DDD$ is assumed to be dualizable, $\Q$-linear, and separated.}

The aim of this article is to establish another instance of Grothendieck duality, namely in the setting of modules over KH-theory in the stable motivic homotopy category. More precisely,  let $\KM$ be the six functor formalism of $\KGL$-modules in $\SH$ \cite[Cor.~13.3.5]{CisinskiDeglise}. For every $S\in\Sch^\fp_B$ over a Noetherian base $B$ there exists an object $\GGL_S$ in $\KM(S)$ representing G-theory (à la Thomason--Trobaugh) such that $f^!\GGL_Y \simeq \GGL_X$ for every morphism $f\colon X\to Y$ of finite type \cite{JinAlgebraicG}. The main result of the present paper is the following.

\begin{introthm}
\label{Thm:introA}
Let $S$ be a finite dimensional, quasi-excellent scheme of characteristic zero, and let $\GGL_S\in \KM(S)$ be the representing object of G-theory. Then $\GGL_S$ is dualizing:
\begin{enumerate}
    \item The canonical map
    \[ \lambda_S \colon \KGL_S \to \Homu_{\KM(S)}(\GGL_S,\GGL_S)\]
is an equivalence in $\KM(S)$ (Theorem~\ref{Thm:lambda-equivalence-characteristic-zero}).
    \item The object $\GGL_S$ is constructible in $\KM(S)$ (Lemma~\ref{Lem:GGL-constrcutible}).
    \item The endofunctor
\[ \KM_c(S)^\op \to \KM_c(S)\colon M \mapsto \Homu_{\KM(S)}(M,\GGL_S) \]
is a self-inverse anti-equivalence (Corollary~\ref{Cor:duality_char0}).
\end{enumerate}
Here, $\KM_c(S)$ denotes the full subcategory of $\KM(S) \coloneqq \Mod_{\KGL_S}(\SH(S))$ spanned by constructible objects. 
\end{introthm}

By an abstract argument on duality (Corollary~\ref{Cor:abstract_duality}),  statement (3) is a purely formal consequence of (1) and (2). The first statement can be shown cdh-locally on stalks, i.e., on the spectra of Henselian valuation rings (Proposition~\ref{Prop:reduction_to_Henselian_valuation_ring}). Using resolutions of singularities we can write any valuation ring as a filtered colimit of regular rings on which K-theory and G-theory agree so that the statement becomes trivial.

\subsection*{Relation to previously known results.}
There are many results in the literature on the question whether $D\coloneqq f^!1_S$  for given $f\colon X\to S$ is dualizing, which entails  that the canonical map
\[ M \to \Homu(\Homu(M,D),D) \]
is an equivalence for every constructible object $M$. To name a few, consider the work of Ayoub \cite[Thm.~2.3.73]{AyoubThesis}, Cisinski--Deglise \cite[Cor.~4.4.24]{CisinskiDeglise}, or Bondarko--Deglise \cite[Thm.~2.4.9]{BondarkoDeglise2017}.
These proofs use that the category of constructible objects is generated by objects coming from regular schemes, see \cite[Prop.~2.2.27]{AyoubThesis}, \cite[Prop.~7.2]{cisinski-deglise-integral-mixed-motives}, and \cite[Cor.~2.4.8]{BondarkoDeglise2017}; this, in turn, is proven by using some form of resolution of singularities (by modifications or by alterations).

\newcommand{\pinv}{[\tfrac{1}{p}]}
Following this approach together with absolute purity one can show Theorem~A for schemes $S$ which are finitely presented over a regular and excellent $\Q$-scheme $B$, using \cite[Thm.~2.4.9]{BondarkoDeglise2017}. In characteristic $p >0$, one can show a version of Theorem~A with $\KGL_S$ replaced by $\KGL_S\pinv$ for $S$ finitely presented over a perfect field, also by \cite[Thm.~2.4.9]{BondarkoDeglise2017}. Although probably known to experts, we review these results in Proposition~\ref{Prop:BD_duality} and Remark~\ref{Rem:purity_duality}.

Our new contribution to the existing literature lies in a strengthening of these known duality results by allowing any quasi-excellent $\Q$-scheme $S$, and in giving a different approach.  We enhance the map $\lambda_S$ to a map of cdh-sheaves (Construction~\ref{Con:underline_lambda}). This allows us to check the statement on stalks at Henselian valuation rings, and these are ind-regular.

\subsection*{Acknowledgements}
The authors cordially thank Denis-Charles Cisinski for proposing the subject of the paper, for giving helpful input, and for valuable discussions and constructive feedback.
Moreover, the authors thank
Fangzhou Jin, Robin de Jong, and Adeel Khan for helpful discussions.

JH has been generously supported by the Knut and Alice Wallenberg Foundation, project number 2021.0287.
The project was jointly supported by the Deutsche Forschungsgemeinschaft (DFG).
Namely, JH is supported by the Collaborative Research Centre SFB 1085 \emph{Higher Invariants -- Interactions between Arithmetic Geometry and Global Analysis}, project number 224262486, and CD is supported by the Collaborative Research Centre TRR 326 \textit{Geometry and Arithmetic of Uniformized Structures (GAUS)}, project number 444845124. Moreover, JH and SW thank GAUS for financing a work stay in Heidelberg in August 2024.
CD thanks Elden Elmanto for hosting him at the University of Toronto in September 2024.

\subsection*{Conventions and notation}
Unless otherwise stated, we assume that all schemes (hence all morphisms) are quasi-compact and quasi-separated (qcqs).
\begin{enumerate}
    \item Write $\PrLo$ for the $\infty$-category of  presentably symmetric monoidal $\infty$-categories, with colimit-preserving, symmetric monoidal functors as morphisms. Write $\PrLotimesst \subset \PrLo$ for the full subcategory spanned by $\CCC \in \PrLo$ which are stable.
   \item For a scheme $S$, let $\Sch_S$ (resp.\ $\Schfp_S$, resp.\ $\Sm_S$) be the category of all (resp.\ finitely presented, resp.\ smooth) $S$-schemes.
    \item Let $\SH(S)$ be the stable motivic homotopy category.
    \item Let $\VVV$ be a closed symmetric monoidal $\infty$-category. For a $\VVV$-enriched $\infty$-category $\CCC$, write $\Hom_\CCC(-,-)_\VVV$ for the mapping object in $\VVV$. For short, we write $\Hom_\CCC(-,-) \coloneqq \Hom_\CCC(-,-)_\SSS$ for the ordinary mapping space. If $\CCC$ is enriched over itself, we write $\Homu_\CCC(-,-) \coloneqq \Hom_\CCC(-,-)_{\CCC}$ for the inner hom. In the case $\VVV = \Sp$ is the $\infty$-category of spectra, we also write $\map(-,-) \coloneqq \Hom_\CCC(-,-)_{\Sp}$.
\end{enumerate}

\section{Duality for coefficient systems}
In this section we briefly explain how dualizing objects lead to duality in the context of  coefficient systems (Definition \ref{Def:coefficient_system}).

Throughout we fix a scheme $B$ and  a  coefficient system 
\[ \DDD \colon \SchfpBop \to \PrLotimesst. \]

\subsection*{Constructible objects}
\begin{Def}
For $S \in \SchfpB$, let $\DDD_c(S)$ be the thick subcategory of $\DDD(S)$ generated by objects of the form $f_\sharp 1_Y(n)$ for all smooth morphisms $f\colon Y\to S$ and all $n\in\Z$.\footnote{Recall that for a stable $\infty$-category $\CCC$ and a class $K$ of objects in $\CCC$, the thick subcategory of $\CCC$ generated by $K$ is the smallest stable subcategory of $\CCC$ which contains $K$ and is closed under retracts.} We say that $M \in \DDD(S)$ is \emph{constructible} if it lies in $\DDD_c(S)$.
\end{Def}

If, for any smooth morphism $f\colon Y\to X$ and any $n\in\Z$, the object $f_\#1_Y(n)$ is compact in $\DDD(X)$, then $\DDD_c(X)$ is precisely the full subcategory of compact objects \cite[Prop.~1.4.11]{CisinskiDeglise}.  This is the case for $\SH$ and more generally for the coefficient system $ \Mod_R(\SH)$ where $R \in \CAlg(\SH)$ as in Notation~\ref{uppershriek_modules} below.\footnote{The latter follows from the fact that the forgetful functor $\Mod_R(\SH) \to \SH$ commutes with filtered colimits \cite[Cor.~4.2.3.5]{LurieHA}.}

\begin{Prop}
\label{Prop:constructive-objects-characterisation}
    For any $S\in\SchfpB$, the category $\DDD_c(S)$ is generated by objects of the form $f_*(1_X)(n)$ for proper (or even projective) morphisms $f\colon X\to S$ and $n\in\Z$. 
\end{Prop}
\begin{proof}
    See \cite[Lem.~2.2.23]{AyoubThesis} or \cite[Prop.~4.2.13]{CisinskiDeglise}.
\end{proof}

\subsection*{(Weakly) dualizing objects}
Now fix an object $I_B \in \DDD(B)$. For a morphism $f \colon S \to B$ of finite presentation, put $I_S \coloneqq f^!I_B$.\footnote{Recall that there is no separatedness assumption needed for the existence of $f^!$ by \cite[Thm.~2.34]{KhanVoevodsky}, following \cite{liu-zheng}.} For $M \in \DDD(S)$ define the \emph{weak dual} as 
$M^\vee \coloneqq \Homu_{\DDD(S)}(M,I_S)$. 
\begin{Def}
Let us call the object $I_S$ \emph{weakly dualizing} if  the canonical map
\[ \lambda_S \colon 1_S \to I_S^\vee = \Homu_{\DDD(S)}(I_S,I_S) \]
is an equivalence.
Furthermore, we call $I_S$ \emph{dualizing} if it is both weakly dualizing and constructible.
\end{Def}
\begin{Prop}
\label{Prop:phiM_equivalence}
The following are equivalent:
    \begin{enumerate}
        \item For every $S \in \SchfpB$ and every $M\in\DDD_c(S)$ the canonical map 
        \[ \phi_M \colon M \to M^{\vee \vee} \]
        is an equivalence.
        \item For every $S \in \SchfpB$ the object $I_S$ is weakly dualizing.
    \end{enumerate}
\end{Prop}
\begin{proof}
    Clearly, if $\phi_M$ is invertible for all $M \in \DDD_c(S)$ then $\lambda_S$ is invertible since $\lambda_S \simeq \phi_{1_S}$. 
    
    For the converse implication note that the full subcategory of $\DDD(S)$ for which $\phi_M$ is an equivalence is a thick subcategory. By Proposition~\ref{Prop:constructive-objects-characterisation}, $\DDD_c(S)$ is the thick subcategory generated by objects of the form $M=f_*(1_X)(n)$ where $f \colon X \to S$ is proper and $n\in\Z$. Thus it suffices to show that $\phi_M$ is an equivalence for such $M$. Since $(M(n))^\vee \simeq M^\vee(-n)$, we may assume $M = f_*(1_X) \simeq f_!(1_X)$.
    Note that for any $N \in \DDD(X)$ it holds
    \[ (f_*N)^{\vee} \simeq \Homu_{\DDD(S)}(f_!N,I_S) \simeq f_* \Homu_{\DDD(X)}(N,I_X) \simeq f_*(N^\vee). \]
    By naturality of $\phi$ it holds that
    \[ \phi_M \simeq f_*(\phi_{1_X}) \colon f_*1_X \to (f_*1_X)^{\vee \vee} \simeq f_*(1_X^{\vee \vee}).  \]
    Since $\lambda_X\simeq \phi_{1_X}$ is assumed an equivalence, the claim follows.
\end{proof}

\begin{Cor}
\label{Cor:abstract_duality}
    Suppose that for all $S\in\SchfpB$ the object $I_S$ is dualizing. Then taking weak duals induces a self-inverse anti-equivalence
    \[ \DDD_c(S)^\op \to \DDD_c(S) \colon M \mapsto M^\vee \]
    for all $S \in \SchfpB$.
\end{Cor}

\begin{proof}
    Let $S \in \SchfpB$ and let $M \coloneqq p_! 1_X$ for $p \colon X \to S$ finitely presented and proper. If $I_X$ is constructible, then also
    \begin{align*}
        M^\vee \simeq \Homu_{\DDD(S)}(p_!1_X,I_S) \simeq p_*\Homu_{\DDD(X)}(1_X,p^!I_S) \simeq 
         p_!I_X
    \end{align*}
     is constructible \cite[Thm.~2.60]{KhanVoevodsky}.  
     Now the claim follows from Proposition~\ref{Prop:phiM_equivalence} and its proof.
\end{proof}

\subsection*{Known duality results}
If $B$ is regular and of finite type over an excellent Noetherian scheme of dimension at most 2 and $I_B = 1_B$, then under certain conditions on the coefficient system $\DDD$, the objects $I_S$ are dualizing for all $S$ separated and of finite type over $B$ \cite[Thm.~4.4.21]{CisinskiDeglise}.  Most notably, $\DDD$ is required to be $\Q$-linear. We are interested in the case where $\DDD$ is the category of $\KGL$-modules, which is not $\Q$-linear. Let us review known results relevant to this case. 

\begin{Not}
\label{uppershriek_modules}
    Suppose that the residue fields of $B$ are all of exponential characteristic $p$. Fix $R_B \in \CAlg(\SH(B))$. Consider the functor
    \begin{align*}
        \MMM \colon \SchfpBop &\to \PrLotimesst \\
        (f \colon S \to B) &\mapsto \Mod_{f^*R_B}(\SH(S))[1/p],
    \end{align*}
    and take $I_B \coloneqq 1_B$. Write also $R_S \coloneqq f^*R_B$, or simply $R$ when the base is clear. Observe that we have canonical equivalences
    \[ \Mod_R(\SH(S))[1/p] \simeq \Mod_{R[1/p]}(\SH(S)) \]
    by the Barr--Beck--Lurie theorem.

    Note that $\MMM$ is a coefficient system for which the forgetful functors 
    \[ U_S \colon \MMM(S) \to \SH(S) \]
    commute with $f^*,f_*$ for any $f$ \cite[Cor.~13.3.3]{CisinskiDeglise}, \cite[Thm.~3.15]{GallauerSix}. It also holds that they commute with $f^!$. Indeed, this follows by passing to left adjoints, since by the projection formula it holds
    \[ f_!((-)\otimes R) \simeq f_!(-) \otimes R. \]
\end{Not}

\begin{Prop}[Bondarko--Deglise]
\label{Prop:BD_duality}
    Suppose that $I_X$ is $\otimes$-invertible for all regular $X \in \SchfpB$, and assume either:
    \begin{enumerate}
        \item $p=1$ and $B$ is regular, excellent, and finite dimensional, or
        \item $p>1$ and $B = \Spec (k)$ for a perfect field $k$.\footnote{One can get rid of the perfectness assumption by the same techniques as used in \cite[Thm.~3.1.1]{Elmanto-Khan-perfection}.}
    \end{enumerate}
    Then $I_S \in \MMM(S)$ is weakly dualizing for all $S \in \SchfpB$. 
\end{Prop}

\begin{proof}
    For $S \in \SchfpB$ the category $\MMM_c(S)$ is generated by objects of the form $f_*(1_X)(n)$, where $f \colon X \to S$ is proper, $X$ is regular, and $n\in \Z$ \cite[Cor.~2.4.8]{BondarkoDeglise2017}. We follow a similar strategy as in the proof of Proposition~\ref{Prop:phiM_equivalence}. Hence it suffices to show that
    \[ \phi_{M} \colon M \to M^{\vee \vee} \]
    is invertible, for each such $f \colon X \to S$ and $M \coloneqq f_*(1_X)$. And again we have
    \[ \phi_M \simeq f_*(\phi_{1_X}) \simeq f_*(\lambda_X) \]
    where $\lambda_X \colon 1_X \to \Homu_{\MMM(S)}(I_X,I_X)$ is the canonical map. Since $X$ is regular, $I_X$ is $\otimes$-invertible by assumption, hence $\lambda_X$ is invertible, and the claim follows.    
\end{proof}

\begin{Rem}
\label{Rem:purity_duality}
    In the setting of Proposition~\ref{Prop:BD_duality} the condition that $I_X$ is $\otimes$-invertible for all regular $X \in \SchfpB$ can be replaced by the condition that $R_S$ satisfies absolute purity in the sense of \cite[Def.~4.3.11]{DegliseFundamental} for all $S \in \SchfpB$. In fact, in this case we conclude that $I_S$ is dualizing for all $S \in \SchfpB$ by \cite[Thm.~2.4.9]{BondarkoDeglise2017}.  

    To see this, it suffices to verify that $\MMM$ satisfies absolute purity in the sense of \cite[Def.~1.3.17]{BondarkoDeglise2017}. By \cite[\S 4.3.4]{DegliseFundamental} it holds that $\MMM$ comes with a system of fundamental classes. The associated purity transformation
    \[ \Sigma^{-N_i} i^* \to i^! \]
    for a regular closed immersion $i \colon Z \to S$ with associated normal bundle $N_i$ between regular schemes in $\SchfpB$ is mapped to the purity transformation in $\SH$.\footnote{This is because the construction of the fundamental classes in question can be done using only the operations $(-)^*, (-)_*, (-)^!$ by \cite[Rem.~3.2.6]{DegliseFundamental}, and these commute with the forgetful functor by Notation~\ref{uppershriek_modules}.} By assumption on $R_S$, the latter is invertible. Since the forgetful functor $\MMM \to \SH$ is conservative, we conclude that the system of fundamental classes in $\MMM$ satisfies the requirements laid out in \cite[Prop.~1.3.21]{BondarkoDeglise2017}.
\end{Rem}

\section{Duality for KGL-modules in characteristic zero}
Throughout this section, let $B$ be a Noetherian base scheme. 

\begin{Rem}[K-theory]
For any  scheme $S$, denote by $\K(S)$ the non-connective K-theory spectrum (à la Thomason--Trobaugh and Blumberg--Gepner--Tabuada) and by $\KH(S)$ its homotopy invariant version due to Weibel \cite{WeibelKH}. Let $\KGL_S \in \SH(S)$ be the representing object of the functor $\KH(-)$ \cite[Thm.~2.20]{CisinskiKH}, \cite[\S 4.2]{KhanKG}; it carries  the structure of an $\E_\infty$-ring \cite{naumannKGLring}. 
Hence we can define $\KM(S) \coloneqq \Mod_{\KGL_S}$ as the $\infty$-category of modules over $\KGL_S$ in $\SH(S)$, see \cite[\S 3.3.3]{LurieHA}. The tensor product on $\SH(S)$ makes $\KM(S)$ into a closed symmetric monoidal $\infty$-category, tensored over $\SH(S)$.
\end{Rem}

\begin{Rem}(G-theory)
For $S\in\SchfpB$, let $\GGL_S \in \SH(S)$ be the representing object of G-theory, as constructed by Jin \cite[Def.~2.2.12]{JinAlgebraicG}. Hence, for every $X\in\Sm_S$ we have a canonical equivalence
    \[ \Ge(X) \simeq \map_{\SH(S)}(\Sigma^\infty_+X,\GGL_S)\]
of spectra with the G-theory spectrum $\Ge(X)$, see \cite[Cor.~2.2.16]{JinAlgebraicG}. Moreover, the object $\GGL_S \in \SH(S)$ carries an action by $\KGL_S$ \cite[2.1.10]{JinAlgebraicG} so that we obtain an object $\GGL_S\in\KM(S)$.

The object $\GGL_S$ has the following description. The presheaf $\Ge \colon \Sm_S^\op \to \Sp$ is an $\A^1$-invariant Nisnevich sheaf. The element $t\in\Z[t,t\inv]$ yields via the morphism $\K_1(\G_m) \to \Ge_1(\G_m)$ a map $t\colon \mathbb{T}\otimes \Ge \to \Ge$ in $\Fun(\Sm_S^\op,\Sp)$ where $\mathbb{T} \coloneqq (S^1,1) \otimes (\G_m,1)$, cf.\ \cite[\S 2.1]{CisinskiKH}. Hence we obtain a $\mathbb{T}$-spectrum $\underline{\Ge}$, cf.\ \cite[\S 2.16]{CisinskiKH}. By design, the object $\underline{\Ge}$ is an $\A^1$-invariant Nisnevich sheaf and hence yields an object $\GGL_S$ in $\SH(S)$. Since G-theory satisfies the Bass Fundamental Theorem, the object $\GGL_S$ represents G-theory.
\end{Rem}

\begin{Rem}
\label{Rem:redefine_GGL}
    For any $f \colon T \to S$ separated of finite type there is an invertible map $\GGL_T \to f^! \GGL_S$ by \cite[Prop.~3.1.10]{JinAlgebraicG}. Tracing through the construction, one can show that this map canonically lifts to a map $\GGL_T \to f^! \GGL_S$ in $\KM(T)$. To see this, one reduces to the case where $f$ is either an open immersion or a proper map. The first case follows from the compatibility of the contravariant functoriality of G-theory with the K-theory action. The proper case follows from the projection formula (e.g., as in \cite[Prop.~3.7]{KhanKG}), which states that the proper pushforward $\Ge(T) \to \Ge(S)$ is $\K(S)$-linear.
    To take care of the higher coherence, one can thus redefine $\GGL_S$ as $g^! \GGL_B \in \KM(S)$ where $g \colon S \to B$ is the structure map.
\end{Rem}

\begin{Def}
Define
\[\Ee_S \coloneqq \Hom_{\KM(S)}(\GGL_S,\GGL_S)_{\SH(S)},\]
 which is the underlying object in $\SH(S)$ of the internal mapping object $\Homu_{\KM(S)}(\GGL_S,\GGL_S)$, induced by the forgetful functor $U_S \colon \KM(S) \to \SH(S)$. Since $\GGL_S$ is a $\KGL_S$-module we have a morphism 
\[  \lambda_S \colon \KGL_S \to \Ee_S\]
in $\SH(S)$, induced by adjunction from $\id_{\GGL_S}$.
\end{Def}

If $B$ is regular, excellent, and finite dimensional, then for any $S \in \SchfpB$ it holds that $\lambda_S$ is invertible by Proposition~\ref{Prop:BD_duality}. In fact, in this case $\GGL_S$ is dualizing by Remark~\ref{Rem:purity_duality} since $\KGL_S$ satisfies absolute purity \cite[\S 13.6]{CisinskiDeglise}.
In this section, we will show a slightly stronger statement, namely that $\GGL_S$ is dualizing for any quasi-excellent $\Q$-scheme $S$ of finite dimension.\footnote{Let $V$ be any quasi-excellent and local $\Q$-algebra of finite dimension which is not universally catenary (such exist, e.g., by \cite[Ex.~2.6]{NishimuraLocal}). Then $\Spec(V)$ is not of finite type over any regular base scheme $B$. Indeed, if it were so then we may assume without loss of generality that $B$ is affine and $\oO_B$ is a regular local ring. Then $\oO_B$, hence also $V$, is universally catenary, contradicting the assumption.}

\begin{Rem}
\label{Rem:lambda_KGL}
    Note that $\lambda_S$ is the image of the map
    \[ \lambda'_S \colon \KGL_S \to \Homu_{\KM(S)}(\GGL_S,\GGL_S) \]
    in $\KM(S)$ under $U_S \colon \KM(S) \to \SH(S)$. Since $U_S$ is conservative, $\lambda_S$ is an equivalence if and only if $\lambda'_S$ is.
\end{Rem}

\subsection*{Reduction to global sections}
\begin{Lem}
\label{Lem:pullback_stability}
    For any $f \colon X \to S$ in $\SchfpB$, it holds $f^*\KGL_S \simeq \KGL_X$. If moreover $f$ is smooth and separated, then also $f^*\Ee_S \simeq \Ee_X$.
\end{Lem}
\begin{proof}
    The statement for $\KGL_S$ is well known \cite[Prop.~3.8]{CisinskiKH}.
    The second statement follows from \cite[Prop.~3.1.11]{JinAlgebraicG} by the fact that in this case
    \[ f^*\Homu_{\KM(S)}(-,-) \simeq \Homu_{\KM(X)}(f^!(-),f^!(-)) \]
    by purity and the projections formulas, together with the fact that $f^*$ and $f^!$ commute with the forgetful functors $U \colon \KM \to \SH$ .   
\end{proof}

\begin{Rem}
    The question whether $\lambda_S$ is an equivalence is Zariski-local on $S$ by Lemma \ref{Lem:pullback_stability}. In particular, we may assume without loss of generality that $B$ is affine, and that all $B$-schemes are separated, which we do from here on. 
\end{Rem}

\begin{Not}
\label{Not:lambdaF}
    For $F \in \SH(S)$, we write
\[ \lambda_S(F) \colon \map_{\SH(S)}(F,\KGL_S) \xrightarrow{(\lambda_S)_*} \map_{\SH(S)}(F,\Ee_S) \]
for the induced map on underlying spectra.
\end{Not}

\begin{Lem}
\label{Lem:reduction_to_global_sections}
    For any smooth $S$-scheme $X$ it holds $\lambda_S(\Sigma^{\infty}_+X) \simeq \lambda_X(1_X)$.
\end{Lem}
\begin{proof}
    By Lemma \ref{Lem:pullback_stability}, $\lambda_S(\Sigma^\infty_+ X)$ is equivalent to the morphism
    \begin{align*}
        \KH(X) \simeq \map_{\SH(S)}(\Sigma^\infty_+ X,\KGL_S) & \to \map_{\SH(S)}(f_{\sharp}1_X,\Ee_S) \\
        & \simeq \map_{\SH(X)}(1_X,f^* \Ee_S) \\
        &\simeq \map_{\SH(X)}(1_X,\Ee_X),
    \end{align*}
    where $f \colon X \to S$ is the structure map. The claim follows.
\end{proof}

\subsection*{Passage to cdh-sheaves}
For any $F \in \SH(B)$, the functor 
\begin{align*}
    \SchfpBop &\to \Sp \\
    (f \colon S \to B) &\mapsto \map_{\SH(S)}(1_S,f^*F)
\end{align*}
satisfies cdh-descent \cite[Lem.~3.4.1]{ehik--milnor-excision}. Since $f^*\KGL_B \simeq \KGL_S$ for any $f \colon S \to B$, one recovers the fact that $S \mapsto \KH(S)$ is a cdh-sheaf. 

\begin{Not}
     Consider the functor
    \begin{align*}
        \Eu_B \colon \SchfpBop &\to \Sp \\
         (f \colon S \to B) & \mapsto \map_{\KM(B)}(f_!f^!\GGL_B,\GGL_B).
    \end{align*}
\end{Not}

\begin{Rem}
\label{Rem:e_S_natural_in_S}
Since $f^!\GGL_S \simeq \GGL_X$ for any $f \colon X \to S$ in $\SchfpB$ by \cite[Prop.~3.1.11]{JinAlgebraicG}, using the exceptional adjunction yields
    \begin{align*}
    \Eu_B(S) 
    \simeq \map_{\KM(S)}(\GGL_S,\GGL_S)
    \simeq \map_{\SH(S)}(1_S,\Ee_S).
    \end{align*}
\end{Rem}

\begin{Lem}
\label{Lem:E_is_cdh_sheaf}
The presheaf $\Eu_B$ is a cdh-sheaf on $\Sch_B^{\fp}$. Moreover, for $f \colon S \to B$ of finite presentation the diagram
    \begin{center}
    \begin{tikzcd}
        \Sch_{S}^{\fp, \op} \arrow[r, "\Eu_{S}"] \arrow[d, "f \circ (-)", swap] & \Sp \\
        \Sch_{B}^{\fp, \op} \arrow[ur, "\Eu_{B}", swap]
     \end{tikzcd}
    \end{center}
    commutes.
\end{Lem}

\begin{proof}
    The functor $\SchfpB \to \KM(B)$ which sends $f \colon S \to B$ to $f_!f^!\GGL_B$
    satisfies cdh-codescent \cite[Cor.~2.50]{KhanVoevodsky}, and $\map_{\KM(B)}(-,\GGL_B)$ sends colimits in $\KM(B)$ to limits of spectra. Since $\Eu_B$ is the composition of these, the first claim follows. The second claim follows from Remark \ref{Rem:e_S_natural_in_S}.  
\end{proof}  

\begin{Cons}
\label{Con:underline_lambda}
We want to construct a natural transformation
\[\underline{\lambda}_B \colon \KH \to \Eu_B \]
of cdh-sheaves such that for any $X \in \SchfpB$ the map $\underline{\lambda}_B(X) \colon \KH(X) \to \Eu_B(X)$ coincides with the map 
\[ \lambda_X(1_X) \colon \Hom_{\SH(X)}(1_X,\KGL_X)_\Sp \xrightarrow{(\lambda_X)_*} \Hom_{\SH(X)}(1_X,\Ee_X)_\Sp \]
from Notation \ref{Not:lambdaF}.

\vspace{6pt}\noindent\textbf{Step 1.}
For $X \in \Sm_B$ we write
\[ \underline{\lambda}_B^\sm(X) \colon \K_{\geq 0}(X) \to \KH(X) \xrightarrow{\lambda_X(1_X)} \Eu_B(X)\]
which gives us a natural transformation $\underline{\lambda}_B^\sm \colon \K_{\geq 0} \to \Eu_B$ of presheaves on $\Sm_B$.

\vspace{6pt}\noindent\textbf{Step 2.}
Write $\Aff_B^{\fp}$ for the category of affine $B$-schemes of finite presentation and put $\Aff_B^{\smt} \coloneqq \Sm_B \cap \Aff_B^{\fp}$. Consider the inclusion functor $j \colon \Aff_B^{\smt} \inj \Aff_B^{\fp}$, and the induced adjunction
\[ j_! \dashv j^* \colon \Fun(\Aff_B^{\smt, \op},\Sp) \rightleftarrows \Fun(\Aff_B^{\fp, \op},\Sp), \]
where $j_!$ is given by left Kan extension.
Since connective algebraic K-theory on $\Aff_B^{\fp}$ is left Kan extended from its restriction to $\Aff_B^{\smt}$ \cite[Ex.~A.0.6(1)]{magic5-modules-cobordism}, we have that $j_!(\K_{\geq 0}) \simeq \K_{\geq 0}$. The map $\underline{\lambda}_B^\sm \colon j^*\K_{\geq 0} \to j^*\Eu_B$ therefore corresponds to a map $\underline{\lambda}_B^{\PPP} \colon \K_{\geq 0} \to \Eu_B$ in $\Fun(\Aff_B^{\fp,\op},\Sp)$ by adjunction. Since $\Eu_B$ is a cdh-sheaf, we obtain an induced map 
\[ \underline{\lambda}_B \colon \KH \simeq \Lcdh\K_{\geq 0} \to \Eu_B \]
in $\Sh_{\cdh}(\Sch_B^{\fp})$.

\vspace{6pt}\noindent\textbf{Step 3.} 
Let $f \colon S \to B$ be finitely presented, and $r \colon \Sch_S^{\fp} \to \Sch_B^{\fp}$ the functor induced by composition with $f$. We claim that the whiskered morphism
\[ \underline{\lambda}_B \cdot r \colon \KH \to \Eu_S \]
of cdh-sheaves on $\Sch_S^{\fp}$ obtained by Lemma~\ref{Lem:E_is_cdh_sheaf} is canonically equivalent to the morphism $\underline{\lambda}_S$. Reversing the previous steps, it suffices to check this after precomposing with $\K_{\geq 0} \to \KH$ and after restricting to $\Sm_S$, in which case it follows by  construction from Lemma~\ref{Lem:reduction_to_global_sections}.

\vspace{6pt}\noindent\textbf{Step 4.} 
It remains to show that $\underline{\lambda}_B$ is pointwise the map $\lambda_X(1_X)$. For $X$ smooth and affine over $B$ this follows from construction. For general $X$ this follows from the previous step.
\end{Cons}

\begin{Prop}
\label{Prop:reduction_to_cdh}
    Suppose that $\underline{\lambda}_B$ is an equivalence. Then for any $S \in \SchfpB$ the map $\lambda_S$ is an equivalence.
\end{Prop}

\begin{proof}  
    Since $\SH(S)$ is generated under colimits by objects of shape $\Sigma^\infty_+X[n]$ for $X$ smooth over $S$ and $n \in \Z$, it suffices to show that $\lambda_S(\Sigma^\infty_+X)$ is an equivalence for all $X \in \Sm_S$. Since $\lambda_S(\Sigma^\infty_+X) = \lambda_X(1_X)$ by Lemma \ref{Lem:reduction_to_global_sections}, the claim follows since $\lambda_X(1_X)$ is $\underline{\lambda}_B(X)$ by construction.
\end{proof}

\subsection*{Reduction to Henselian valuation rings}
Recall that the points for the cdh-topology on $\SchfpB$ are all maps of the form $x \colon \Spec(V) \to B$ with $V$  a Henselian valuation ring \cite[Thm.~2.6]{GabberPoints}. For a cdh-sheaf $F \colon \SchfpBop \to \DDD$ with values in a presentable $\infty$-category $\DDD$, the \emph{stalk} of $F$ at such a point $x\colon \Spec(V) \to B$ is the colimit
\[ F_x \coloneqq \colim_{X \in\mathrm{Nbh}(x)} F(X), \]
where $\mathrm{Nbh}(x)$ is the category of all factorizations $\Spec(V) \to X \to B$ of $x$ such that $X \to B$ is finitely presented.

\begin{Prop}
\label{Prop:reduction_to_Henselian_valuation_ring}
 Suppose that $B$ is of finite dimension. If for every Henselian valuation ring $V$ and every morphism $x\colon\Spec(V)\to B$, the induced map on stalks   
    \[ \lambda_x \colon \KH_x \to (\Eu_B)_x \]
is an equivalence, then $\lambda_S$ is an equivalence for all $S \in \SchfpB$.
\end{Prop}
\begin{proof}
It suffices to show that $\underline{\lambda}_B$ is an equivalence by Proposition \ref{Prop:reduction_to_cdh}.

Since $S$ is Noetherian, the Krull-dimension equals the valuative dimension \cite[Prop.~2.3.2]{ehik--milnor-excision}, and hence the topos of cdh-sheaves on $\SchfpB$ is hypercomplete \cite[Thm.~B]{ehik--milnor-excision}. Hence $\underline{\lambda}_B$ is an equivalence if and only if it induces an equivalence on the homotopy groups if and only if it induces an equivalence on all stalks (as there are enough points) \cite[Thm.~0.2, Thm.~2.6]{GabberPoints}.
\end{proof}

\subsection*{Conclusion of proof}
In this subsection we show $\GGL_S$ is dualizing, assuming resolution of singularities on $S$.

\begin{Thm}
\label{Thm:lambda-equivalence-characteristic-zero}
Suppose that $S$ is a quasi-excellent $\Q$-scheme of finite dimension.\footnote{In fact, the statement is true for any scheme $S$ of finite dimension such that any integral finite type $S$-scheme admits a desingularization.}
Then the canonical map 
\[ \lambda_S \colon \KGL_S \to \Ee_S \]
in $\SH(S)$ is an equivalence.
\end{Thm}
\begin{proof}
According to Proposition~\ref{Prop:reduction_to_Henselian_valuation_ring} it suffices to check the desired assertion on stalks of $\underline{\lambda}_S$ in the cdh-topology. 
Thus let $x\colon \Spec(V) \to S$ be a morphism for a Henselian valuation ring $V$. 

\vspace{6pt}\noindent\textbf{Claim:} $V$ is isomorphic to the filtered colimit of its subrings that are regular and of finite type over $S$.

\vspace{6pt} By Zariski descent we may assume without loss of generality that $S$ is affine, say $S = \Spec(R)$.
Let $W \subset V$ be a sub-$R$-algebra of finite type so that $\Spec(W)$ is integral. By resolution of singularities in characteristic zero, there exists a proper and birational morphism $\pi\colon \tilde{X}\to\Spec(W)$ with $\tilde{X}$ regular which is an isomorphism over the regular locus of $\Spec(W)$ \cite[Thm.~1.1]{temkin-desingularization-char-zero}. 
 Since $\pi$ is proper and surjective, the valuative criterion of properness yields a lift $\tilde{x}\colon \Spec(V)\to\tilde{X}$. Let $\Spec(\tilde{W})$ be an affine open neighbourhood of $\tilde{X}$ containing the image of $\tilde{x}$. Note that $\tilde{W}$ is integral. We have morphisms of $R$-algebras $W\subset\tilde{W}\to V$ with $\tilde{W}$ being regular. On fraction fields this yields $\Frac(W) = \Frac(\tilde{W}) \subset \Frac(V)$, hence $\tilde{W} \to V$ is injective as well. By cofinality, the Claim follows.

Let $I$ be the full subcategory of the comma category $\Spec(V)/\Sch_S^{\fp}$ consisting of spectra of all $R$-subalgebras of $V$ which are regular. By the claim it holds that $I$ is cofiltered and $\Spec(V) \simeq \lim_{S_\alpha \in I} S_\alpha$. It follows that the inclusion $F \colon I \to \Spec(V)/\Sch_S^{\fp}$ is initial. Indeed, let $Y \in \Spec(V)/\Sch_S^{\fp}$ be given. Then
\[ \Hom_S(\Spec(V),Y) \simeq \colim_{S_\alpha \in I} \Hom_S(S_\alpha,Y), \]
from which we conclude that the comma category $F/Y$ is non-empty and filtered, hence contractible. The claim follows by Quillen's Theorem A.

The stalk at $x \colon \Spec(V) \to S$ of any cdh-sheaf $F$ on $\Sch_S^{\fp}$ can thus be computed as 
\[ F_x \simeq \colim_{S_\alpha \in I} F(S_\alpha). \]
It thus suffices to show that
$ \underline{\lambda}_S(S_\alpha) \simeq \lambda_{S_\alpha}(S_\alpha) $
is an equivalence for any $S_\alpha \in I$. And indeed this is so, since $\KGL_T \simeq \GGL_T$ for regular $T$ \cite[Lem.~3.3.3]{JinAlgebraicG}.
\end{proof}

\begin{Lem}
\label{Lem:GGL-constrcutible}
Let $S$ be a quasi-excellent $\Q$-scheme of finite dimension. Then $\GGL_S$ is constructible in $\KM(S)$.
\end{Lem}
\begin{proof}
    We induct on the dimension of $S$. If $\dim(S)=0$ then $S$ is a disjoint union of finitely many schemes of the form $\Spec(A)$ for $A$ an Artinian local ring of dimension zero \cite[\href{https://stacks.math.columbia.edu/tag/0AAX}{Tag 0AAX}]{stacks-project}. By Zariski descent we assume $S = \Spec(A)$, and by nilinvariance that $S$ is reduced. In this case $A$ is a field.
    
    Now assume $\dim (S) > 0$. Let $i \colon Z \to S$ be any irreducible component of $S$ and let $f \colon X \to S$ be the union of the other (finitely many) irreducible components; let $g\colon E\coloneqq Z\times_SX \to S$ be the obvious morphism. By proper excision \cite[Thm.~2.47]{KhanVoevodsky} we have a Cartesian square
    \begin{center}
        \begin{tikzcd}
            \GGL_S \arrow[r] \arrow[d] & i_!\GGL_Z \arrow[d] \\
            f_!\GGL_X \arrow[r] & g_!\GGL_E
        \end{tikzcd}
    \end{center}
    in $\KM(S)$. Note that exceptional pushforward along finite type maps preserves constructible objects  \cite[Thm.~2.60]{KhanVoevodsky}. Hence, by induction on the number of irreducible components and induction on the dimension, we see that $\GGL_S$ is constructible, if the integral case is known. Thus we may assume $S$ to be integral.
    
    By resolution of singularities there exists a blow-up $f' \colon X' \to S$ with $X'$ being regular with center $i' \colon Z' \to S$ such that $\dim(Z') < \dim(S)$ \cite[Thm.~1.1]{temkin-desingularization-char-zero}. Let $g'\colon E'\coloneqq Z' \times_S X' \to S$ be the obvious morphism; note that $\dim(E') < \dim(X') = \dim(S)$.
    Since $X'$ is regular, it holds that $\GGL_{X'}\simeq \KGL_{X'}$ is constructible in $\KM(X')$. By induction on the dimension, we may assume also $\GGL_{Z'}$ and $\GGL_{E'}$ to be constructible. Hence we conclude, by again using the proper excision square above, that $\GGL_S$ is constructible.
\end{proof}

\begin{Cor}
\label{Cor:duality_char0}
Let $S$ be a quasi-excellent $\Q$-scheme of finite dimension. Then taking weak duals gives a self-inverse anti-equivalence 
    \[  \KM_c(S)^\op \to \KM_c(S)\colon M \mapsto M^\vee \coloneqq \Homu_{\KM(S)}(M,\GGL_S)  \]
on the category of constructible objects $\KM_c(S)$.
\end{Cor}
\begin{proof}
    By Theorem~\ref{Thm:lambda-equivalence-characteristic-zero}, Remark~\ref{Rem:lambda_KGL} and Corollary~\ref{Cor:abstract_duality}, this follows from Lemma~\ref{Lem:GGL-constrcutible}.
\end{proof}


\appendix
\section{Coefficient systems}
\label{Sec:coefficient_systems}
Here we review some of the basic theory of coefficient systems. Since there are different presentations available in the literature with various degrees of generality, this appendix is mainly intended to fix notation and to explain how to pass to some of the other approaches so that we can  use results formulated therein. 

Throughout, let $S$ be a scheme.  Consider a full subcategory $\CCC \subset \Sch_S$ such that the following holds:
\begin{enumerate}
    \item For any closed immersion $Z \to X$, if $X \in \CCC$ then $Z_\red \in \CCC$.
    \item The category $\CCC$ has finite coproducts and fiber products.
    \item For any smooth morphism $X \to Y$ in $\Sch_S$, if $Y \in \CCC$ then $X \in \CCC$.
\end{enumerate}

\begin{Def}
\label{Def:coefficient_system}
    A \emph{coefficient system on $\CCC$} is a functor
    \[ \DDD \colon \CCC^\op \to \PrLotimesst \]
    where for $f \colon X \to Y$ we write $f^*$ for the colimit-preserving, symmetric monoidal functor $\DDD(f) \colon \DDD(Y) \to \DDD(X)$, and $f_*$ for its right adjoint. We require the following axioms:
    \begin{itemize}[align=left]
        \item[\textbf{Sharp.}] For any smooth morphism $p \colon X \to Y$ in $\CCC$, the functor $p^*$ admits a left adjoint $p_\sharp$ which satisfies:
        \begin{itemize}[align=left]
            \item [\textit{Smooth base change.}] For any Cartesian square
        \begin{center}
            \begin{tikzcd}
                X' \arrow[r, "g"] \arrow[d, "q"] & X \arrow[d, "p"] \\
                Y' \arrow[r, "f"] & Y
            \end{tikzcd}
        \end{center}
        in $\CCC$, the Beck--Chevalley transformation $q_\sharp g^* \to f^* p_\sharp$ is an equivalence.
        \item [\textit{Projection formula.}] The functor $p_\sharp$ is $\DDD(Y)$-linear. In particular, the canonical map $p_\sharp(p^*A \otimes B) \to p^*A \otimes p_\sharp(B)$ is an equivalence for any $A \in \DDD(Y)$ and $B \in \DDD(X)$. 
        \end{itemize}        
        \item [\textbf{Localization.}] It holds $\DDD(\varnothing) = 0$, and for any closed immersion $i \colon Z\to X$ in $\CCC$ with complementary open immersion $j \colon U \to X$, the sequence
        \[ \DDD(Z) \xrightarrow{i_*} \DDD(X) \xrightarrow{j^*} \DDD(U) \]
        is exact.
        \item [\textbf{Motivic conditions.}] 
        For any $X \in \CCC$ and vector bundle $p \colon E \to X$ with zero section $s \colon X \to E$, it holds:
        \begin{itemize}[align=left]
            \item [\textit{Homotopy invariance.}] $p^*$ is fully faithful.
            \item [\textit{Thom stability.}] $p_\sharp s_*$ is invertible.
        \end{itemize}    
    \end{itemize}
\end{Def}

Let $\CCC_\red$ be the full subcategory of $\Sch_S$ spanned by objects of the form $X_\red$ for $X \in \CCC$. Note that $\CCC_\red$ is a full subcategory of $\CCC$, and the inclusion $\CCC_\red \to \CCC$ admits a right adjoint $(-)_\red$. Restricting a coefficient system on $\CCC$ along $\CCC_\red \to \CCC$ induces a coefficient system on $\CCC_\red$, and any coefficient system on $\CCC_\red$ extends uniquely to a coefficient system on $\CCC$. The latter follows from nilinvariance, which is a consequence of localization.

To apply results from \cite{KhanVoevodsky}, let $\dSch_S$ be the $\infty$-category of \emph{derived} schemes over $S$, and let $\CCC'$ be smallest full subcategory of $\dSch_S$ which contains $\CCC$, is closed under fiber products in $\dSch_S$, and such that for any closed immersion $Z \to X$ with $X \in \CCC'$ it holds $Z \in \CCC'$. Note that taking the underlying reduced classical scheme gives a right adjoint to the inclusion $\CCC_\red \subset \CCC'$. With a derived analogue of above axioms for $\CCC'$ we can likewise pass back and forth between coefficient systems on $\CCC$ and coefficient systems on $\CCC'$. Moreover, the following holds:

\begin{Lem}\label{Lem:CoefsysStar}
    Let $\DDD$ be a coefficient system on $\CCC$ with extension $\DDD'$ to $\CCC'$. Then $\DDD'$ is a $(*,\sharp,\otimes)$-formalism which satisfies the Voevodsky conditions in the sense of \cite[Def.~2.2, Def.~2.4]{KhanVoevodsky}.
\end{Lem}

\begin{proof}
    Additivity follows from localization, and the other axioms are immediate.
\end{proof}

\begin{Rem}
    Suppose that $S$ is Noetherian and finite dimensional, and take $\CCC$ to be the category of schemes which are separated and finite type over $S$. Then a coefficient system on $\CCC$ in the above sense is the same thing as a presentable coefficient system over $S$ in the sense of \cite[Def.~2.4]{GallauerSix}. See also \cite[Rem.~2.21]{GallauerSix} for a comparison with other approaches in the literature. For instance, passing to homotopy categories yields a \emph{motivic triangulated category} over $\CCC$ in the sense of Cisinski--Deglise \cite[Def.~2.4.45]{CisinskiDeglise}.
\end{Rem}

\begin{Not}
    Let $\DDD$ be a coefficient system on $\CCC$, and let $X \in \CCC$ and $E$ a vector bundle on $X$. Write $p \colon E \to X$ for the projection and $s \colon X \to E$ for the zero section. Then the \emph{Thom twist} is the functor
    \[ \Sigma^E \coloneqq p_\sharp s_* \colon \DDD(X) \to \DDD(X). \]
    We write $\langle n \rangle \coloneqq \Sigma^{\A^n_X}$ for $n \in \N$, and let $\langle -n \rangle$ be its inverse. The \emph{Tate twist} is then the functor
    \[ (n) \coloneqq \langle n \rangle [-2n]. \]
\end{Not}

\bibliographystyle{dary} 
\bibliography{refs}	

\newcommand{\etalchar}[1]{$^{#1}$}
\providecommand{\MR}{\relax\ifhmode\unskip\space\fi MR }
\providecommand{\MRhref}[2]{%
  \href{http://www.ams.org/mathscinet-getitem?mr=#1}{#2}
}
\providecommand{\href}[2]{#2}
\begin{thebibliography}{EHK{\etalchar{+}}20}

\bibitem[Ayo07]{AyoubThesis}
Joseph Ayoub, \emph{Les six op\'{e}rations de {G}rothendieck et le formalisme
  des cycles \'{e}vanescents dans le monde motivique. {(I, II)}},
  Ast\'{e}risque (2007), no.~314, 315.

\bibitem[BD17]{BondarkoDeglise2017}
Mikhail Bondarko and Fr\'ed\'eric D\'eglise, \emph{Dimensional homotopy
  t-structures in motivic homotopy theory}, Adv. Math. \textbf{311} (2017),
  91--189.

\bibitem[CD15]{cisinski-deglise-integral-mixed-motives}
Denis-Charles Cisinski and Fr\'ed\'eric D\'eglise, \emph{Integral mixed motives
  in equal characteristic}, Doc. Math. (2015), 145--194.

\bibitem[CD19]{CisinskiDeglise}
Denis-Charles Cisinski and Fr\'{e}d\'{e}ric D\'{e}glise, \emph{Triangulated
  categories of mixed motives}, Springer Monographs in Mathematics, Springer,
  Cham, 2019.

\bibitem[Cis13]{CisinskiKH}
Denis-Charles Cisinski, \emph{Descente par \'eclatements en {$K$}-th\'eorie
  invariante par homotopie}, Ann. of Math. (2) \textbf{177} (2013), no.~2,
  425--448.

\bibitem[DJK21]{DegliseFundamental}
Fr{\'e}d{\'e}ric D{\'e}glise, Fangzhou Jin, and Adeel~A. Khan,
  \emph{Fundamental classes in motivic homotopy theory}, Journal of the
  European Mathematical Society \textbf{23} (2021), no.~12, 3935--3993.

\bibitem[EHIK21]{ehik--milnor-excision}
Elden Elmanto, Marc Hoyois, Ryomei Iwasa, and Shane Kelly, \emph{Cdh descent,
  cdarc descent, and {M}ilnor excision}, Math. Ann. \textbf{379} (2021),
  no.~3-4, 1011--1045.

\bibitem[EHK{\etalchar{+}}20]{magic5-modules-cobordism}
Elden Elmanto, Marc Hoyois, Adeel~A. Khan, Vladimir Sosnilo, and Maria
  Yakerson, \emph{Modules over algebraic cobordism}, Forum Math. Pi \textbf{8}
  (2020), e14, 44.

\bibitem[EK20]{Elmanto-Khan-perfection}
Elden Elmanto and Adeel~A. Khan, \emph{Perfection in motivic homotopy theory},
  Proc. Lond. Math. Soc. (3) \textbf{120} (2020), no.~1, 28--38.

\bibitem[Gal21]{GallauerSix}
Martin Gallauer, \emph{An introduction to six-functor formalisms}, 2021,
  \href{http://arXiv.org/abs/2112.10456v1}{\mbox{arXiv:2112.10456v1}}.

\bibitem[GK15]{GabberPoints}
Ofer Gabber and Shane Kelly, \emph{Points in algebraic geometry}, J. Pure Appl.
  Algebra \textbf{219} (2015), no.~10, 4667--4680.

\bibitem[Jin19]{JinAlgebraicG}
Fangzhou Jin, \emph{Algebraic {G}-theory in motivic homotopy categories}, 2019,
  \href{http://arXiv.org/abs/1806.03927v3}{\mbox{arXiv:1806.03927v3}}.

\bibitem[Kha21]{KhanVoevodsky}
Adeel~A. Khan, \emph{Voevodsky's criterion for constructible categories of
  coefficients}, Date: 2021-01-25 (some revisions on 2023-03-17)
  \url{https://www.preschema.com/papers/six.pdf}, 2021.

\bibitem[Kha22]{KhanKG}
Adeel~A. Khan, \emph{K-theory and {G}-theory of derived algebraic stacks}, Jpn.
  J. Math. \textbf{17} (2022), no.~1, 1--61.

\bibitem[Lur17]{LurieHA}
Jacob Lurie, \emph{Higher {A}lgebra}, (version dated September 18, 2017)
  \url{https://www.math.ias.edu/~lurie/papers/HA.pdf}, 2017.

\bibitem[LZ24]{liu-zheng}
Yifeng Liu and Weizhe Zheng, \emph{Enhanced six operations and base change
  theorem for higher {A}rtin stacks}, 2024,
  \href{http://arXiv.org/abs/1211.5948}{\mbox{arXiv:1211.5948}}.

\bibitem[Nis12]{NishimuraLocal}
Jun-ichi Nishimura, \emph{A few examples of local rings, {I}}, Kyoto J. Math.
  \textbf{52} (2012), no.~1, 51--87.

\bibitem[NS{\O}15]{naumannKGLring}
Niko Naumann, Markus Spitzweck, and Paul~Arne {\O}stv{\ae}r, \emph{Existence
  and uniqueness of {$E_\infty$} structures on motivic {$K$}-theory spectra},
  J. Homotopy Relat. Struct. \textbf{10} (2015), no.~3, 333--346.

\bibitem[{Sta}]{stacks-project}
The {Stacks Project Authors}, \emph{\textit{Stacks Project}},
  \url{http://stacks.math.columbia.edu}.

\bibitem[Tem08]{temkin-desingularization-char-zero}
Michael Temkin, \emph{Desingularization of quasi-excellent schemes in
  characteristic zero}, Adv. Math. \textbf{219} (2008), no.~2, 488--522.

\bibitem[Wei89]{WeibelKH}
Charles~A. Weibel, \emph{Homotopy algebraic {$K$}-theory}, Algebraic
  {$K$}-theory and algebraic number theory ({H}onolulu, {HI}, 1987), Contemp.
  Math., vol.~83, Amer. Math. Soc., Providence, RI, 1989, pp.~461--488.

\bibitem[Zav23]{ZavyalovPD}
Bogdan Zavyalov, \emph{Poincar\'e duality in abstract 6-functor formalisms},
  2023, \href{http://arXiv.org/abs/2301.03821v2}{\mbox{arXiv:2301.03821v2}}.

\end{thebibliography}

\end{document}